\documentclass{amsart}
\usepackage[T1]{fontenc}
\usepackage[utf8]{inputenc}
\usepackage{amsthm}
\usepackage{amssymb}

\makeatletter

\numberwithin{equation}{section}
\numberwithin{figure}{section}

\usepackage{latexsym}
\usepackage{amssymb}
\usepackage{amsmath}
\usepackage{enumerate}
\usepackage{amsmath,amsthm,amsfonts,amssymb,latexsym,xy}
\xyoption{all}

\usepackage{latexsym}
\usepackage{amsthm}
\usepackage{amssymb}
\usepackage{amsmath}
\usepackage{amsfonts}
\usepackage{mathrsfs}
\usepackage{tikz}
\usetikzlibrary{matrix}
\usepackage[T1]{fontenc}
\usepackage[utf8]{inputenc}

\theoremstyle{definition}

\newtheorem{theorem}{Theorem}
\newtheorem{lemma}[theorem]{Lemma}

\newtheorem{question}[theorem]{Question}

\makeatother

\begin{document}

\title{Infinitely supported Liouville measures of Schreier graphs}
\author{Kate Juschenko}

\address{Kate Juschenko, Department of Mathematics, Northwestern University,
2033 Sheridan Road, Evanston, IL 60208, USA}

\email{kate.juschenko@gmail.com}

\author{Tianyi Zheng}

\address{Tianyi Zheng, Department of Mathematics, UC San Diego, 9500 Gilman
Dr, La Jolla, CA 92093, USA}

\email{zhengty04@gmail.com}

\date{\today}

\maketitle

\begin{abstract}
We provide equivalent conditions for Liouville property of actions of
groups. As an application, we show that there is a Liouville measure
for the action of the Thompson group $F$ on dyadic rationals. This
result should be compared with a recent result of Kaimanovich, where
he shows that the action of the Thompson group F on dyadic rationals
is not Liouville for all finitely supported measures. As another application
we show that there is a Liouville measure for lamplighter actions.
This gives more examples of non-amenable Liouville actions.
\end{abstract}

\section{introduction}

Let a discrete group $G$ act on a countable set $X$, denoted by
$G\curvearrowright X$, and let $\mu$ be a probability measure on
$G$. A measure $\mu$ on $G$ is non-degenerate if $\mbox{supp}\mu$
generate the group $G$. We denote by $P_{\mu}$ the transition matrix
on $X$ induced by $\mu$, that is 
\[
P_{\mu}(x,y)=\sum_{g\in G}\mathbf{1}_{\{g\cdot x=y\}}\mu(g).
\]
For the simplicity of the notations we write $\mu\cdot x=P_{\mu}(x,\cdot)$.
A function $f:X\to\mathbb{R}$ is $P_{\mu}$-harmonic if $f(x)=\sum_{y\in X}f(y)P_{\mu}(x,y)$,
and $(X,P_{\mu})$ is Liouville if all bounded $P_{\mu}$-harmonic
functions are constant. The action $G\curvearrowright X$ is $\mu$-Liouville
if $(X,P_{\mu})$ is Liouville, and if this is the case we say $\mu$
is a Liouville measure for the action. We call an action Liouville
if there is a measure $\mu$ on $G$ which makes it $\mu$-Liouville.

Note that, one can make several definitions for Liouville actions
by adapting the definition of Liouville measures on Cayley graphs.
While many of the definitions are equivalent for Cayley graphs, they
are not equivalent when we pass to actions. The main reason of our
current definition is a recent approach to amenability developed by
Kaimanovich in \cite{Kaimanovich}. In order to show that a group
is not amenable it is sufficient to find an action which does not
admit any non-degenerate Liouville measure. Indeed, this will insure
that there is no non-denegerate Liouville measure on the group itself,
thus, by renowned result of Kaimanovich and Vershik, \cite{KaimVersh},
the group is not amenable. The problem of amenability of Thompson's
group F can be approached with this technique. In particular, Kaimanovich
showed that every finitely supported non-degenerate measure on Thompson
group F is not Liouville for it's action on dyadic rationals. In this
paper we show that this action admits infinitely supported Liouville
measure.

In fact, our methods are more general, we give a criteria for a measure
to be Liouville for the action. As another application of this method
we show that certain Schreier graphs of lamplighter group are non-amenable
but $\mu$-Liouville for some infinitely supported measure $\mu$.
We note that examples of non-amenable graphs with Liouville measures
of finite support were previously have been known, see for example \cite{Ben-Koz} and references therein.\\

{\bf Acknowledgements:} We are grateful to Kaimanovich for several interesting and motivating discussions.   We also thank Nico Matte Bon and Omer Tamusz for remarks on earlier versions of the paper.

\section{Liouville measures on Schreier graphs}

We fix a finite generating set $S$ on $G$. Consider an action of
$G$ on a countable set $X$. Let $\left\Vert \cdot\right\Vert _{1}$
be the $\ell^{1}$ norm with respect to counting measure on $X$,
that is $\left\Vert f\right\Vert _{1}=\sum_{x\in X}|f(x)|$. 

\begin{lemma}

Suppose $P_{\mu}$ is irreducible and lazy such that $P_{\mu}(x,x)\ge\frac{1}{2}$
for all $x\in X$. If $(X,P_{\mu})$ is Liouville, then for any two
points $x,y\in X$, 
\[
\lim_{n\to\infty}\left\Vert \mu^{(n)}\cdot x-\mu^{(n)}\cdot y\right\Vert _{1}=0.
\]

\end{lemma}

\begin{proof}

Since $P_{\mu}$ is irreducible, it is sufficient to prove the claim
for $y=s\cdot x$ with $s\in\mbox{supp}(\mu)$. Let $\left(W_{n}\right)$
be a left random walk on $G$ with step distribution $\mu$.

By Corollary 14.13 of \cite{LP:book}, $(X,P_{\mu})$ is Liouville
if and only if the invariant $\sigma$-field $\mathcal{I}$ is trivial.
Since $P_{\mu}$ is lazy, by Theorem 14.18 \cite{LP:book}, the completion
of the tail $\sigma$-field coincides with the completion of the invariant
$\sigma$-field. Thus Liouville property of $(X,P_{\mu})$ implies
that the tail $\sigma$-field $\mathcal{T}$ is trivial. Therefore,
\[
\mathbb{P}(W_{1}\cdot x=y_{1}|W_{n}\cdot x)\to\mathbb{P}(W_{1}\cdot x=y_{1})
\]
almost surely when $n\to\infty$. Since 
\[
\mathbb{P}(W_{1}\cdot x=y_{1}|W_{n}\cdot x=y_{n})=\frac{P_{\mu}(x,y_{1})P_{\mu}^{n-1}(y_{1},y_{n})}{P_{\mu}^{n}(x,y_{n})},
\]
we have for $y_{1}=s\cdot x$, 
\[
\lim_{n\to\infty}\frac{P_{\mu}^{n-1}(s\cdot x,W_{n}\cdot x)}{P_{\mu}^{n}(x,W_{n}\cdot x)}=1\ a.s
\]
Then it implies for any $\epsilon>0$, 
\[
\lim_{n\to\infty}\mu^{(n)}\left(\left\{ g:\ \left|\frac{P_{\mu}^{n-1}(s\cdot x,g\cdot x)}{P_{\mu}^{n}(x,g\cdot x)}-1\right|\ge\epsilon\right\} \right)=0.
\]
It follows that 
\[
\lim_{n\to\infty}\left\Vert P_{\mu}^{n-1}(s\cdot x,\cdot)-P_{\mu}^{n}(x,\cdot)\right\Vert _{1}=0.
\]
Finally, by Theorem 14.16 in \cite{LP:book}, laziness of $P_{\mu}$
guarentees that 
\[
\left\Vert P_{\mu}^{n-1}(s\cdot x,\cdot)-P_{\mu}^{n}(s\cdot x,\cdot)\right\Vert _{1}\to0
\]
when $n\to\infty$.

\end{proof}

In the other direction, we can build a measure $\mu$ on $G$ such
that $P_{\mu}$ is Liouville from pieces that have good coupling properties.

\begin{lemma}\label{mainlemma}

Suppose there exists an increasing sequence of finite subsets $\left(K_{n}\right)$
exhausting $X$ and a sequence $\left(\epsilon_{n}\right)$ decreasing
to $0$ such that for each $n$, there exists a probability measure
$\nu_{n}$ of finite support on $G$ such that for any $x,y\in K_{n}$
such that $y=s\cdot x$ for some $s\in S$, we have 
\[
\left\Vert \nu_{n}\cdot x-\nu_{n}\cdot y\right\Vert _{1}<\epsilon_{n}.
\]
Then there exists a non-degenerate probability measure $\mu$ on $G$
such that $(X,P_{\mu})$ is Liouville.

\end{lemma}

\begin{proof}

Our proof is reminiscent to the proof of Theorem 4.3 in \cite{KaimVersh}.
The measure $\mu$ is obtained as a convex combination of a subsequence
of $\left(\nu_{n}\right)$, 
\[
\mu=\sum_{j=0}^{\infty}c_{j}\zeta_{j},\ \zeta_{j}=\nu_{n_{j}}
\]
To make $\mu$ non-degenerate, we take $\nu_{0}$ to be uniform on
$S\cup S^{-1}$.

First note that to show $P_{\mu}$ is Liouville, it suffices to show
for any $x,y\in X$ connected by an edge, $y=s\cdot x$ for some $s\in S$,
we have 
\[
\liminf_{m\to\infty}\left\Vert \mu^{(m)}\cdot x-\mu^{(m)}\cdot y\right\Vert _{1}=0.
\]
(This implies that for any bounded $P_{\mu}$-harmonic function $h$,
$h(x)=h(y)$ for neighboring points, thus the function must be constant.)

Since each $\nu_{n}$ is assumed to be of finite support on $G$,
let $B(e,r_{n})$ be a ball large enough on $G$ such that $\mbox{support}\nu_{n}\subset B(e,R_{n})$.
Let $r_{n}$ be the largest ball such that $B(o,r_{n})\subseteq K_{n}$,
we have $r_{n}\to\infty$ since $\left(K_{n}\right)$ exhausts $X$.
Fix a consequence of weights $\left(c_{j}\right)$, select $(n_{j})$
inductively as follows: let $m_{j}$ be the smallest integer such
that $\left(c_{0}+\ldots+c_{j-1}\right)^{m_{j}}\le1/j$, take $n_{j}$
to be the least integer such that 
\[
m_{j}R_{n_{j-1}}\le r_{n_{j}}\mbox{ and }m_{j}R_{n_{j-1}}\epsilon_{n_{j}}\le1/j.
\]

For $m$-th convolution power of $\mu$, 
\[
\mu^{(m)}=\sum_{k}c_{k_{1}}\ldots c_{k_{m}}\zeta_{k_{m}}\ast\ldots\ast\zeta_{k_{1}}.
\]
Consider two parts, $\mu_{1}^{(m)}$ consists these terms with $\max_{1\le i\le m}k_{i}\ge j$,
and $\mu_{2}^{(m)}=\mu^{(m)}-\mu_{1}^{(m)}$. For $m=m_{j}$, the
total mass of the first part is 
\[
\left\Vert \mu_{1}^{(m_{j})}\right\Vert _{1}=\left(c_{1}+\ldots+c_{j-1}\right)^{m_{j}}.
\]
For each term in the second part, let $i=i(k)$ be the lowest index
such that $k_{i}\ge j$. Starting at two neighboring points $x,y$,
consider the distribution induced by $\zeta_{k_{i}}\ast\ldots\ast\zeta_{k_{1}}$
on these two points. Since for $i<k$, $k_{i}<j$, it follows that
the support of $\zeta_{k_{i-1}}\ast\ldots\ast\zeta_{k_{1}}\cdot x$
and $\zeta_{k_{i-1}}\ast\ldots\ast\zeta_{k_{1}}\cdot y$ are contained
in the ball $B_{X}\left(o,d(o,x)+1+(m_{j}-1)R_{n_{j-1}-1}\right)$.
By the choice of the $\left(\zeta_{n}\right)$, we have for $k_{i}\ge j$,
\[
\left\Vert \zeta_{k_{i}}\ast\zeta_{k_{i-1}}\ast\ldots\ast\zeta_{k_{1}}\cdot x-\zeta_{k_{i}}\ast\zeta_{k_{i-1}}\ast\ldots\ast\zeta_{k_{1}}\cdot y\right\Vert _{1}\le2(m_{j}-1)R_{n_{j-1}-1}\epsilon_{n_{j}}.
\]
Combine the two parts, we have 
\[
\left\Vert \mu^{(m_{j})}\cdot x-\mu^{(m_{j})}\cdot y\right\Vert _{1}\le\left(c_{1}+\ldots+c_{j-1}\right)^{m_{j}}+2(m-1)R_{\ell_{m}-1}\epsilon_{\ell_{m}}\le3/j
\]

\end{proof}

\section{Applications to the Thompson group F}

We denote by $F$ the Thompson group $F$. In \cite{Kaimanovich}
and \cite{Mishchenko}, authors show that the Schreier graph of the
action of $F$ on the orbit of $1/2$ is not Liouville with respect
to any measure of finite support. Here we show that there are measures
with infinite support that make this action Liouville. In fact, one
can even choose symmetric ones.

\begin{theorem} There is a non-degenerate symmetric measure $\mu$
on Thompson group $F$ such that the action on $Orb(1/2)$ is $\mu$-Liouville.
\end{theorem}

\begin{proof} By Lemma \ref{mainlemma}, we have to find a measure
that approximates any finite subset in $Orb(1/2)$. The Schreier graph
of $Orb(1/2)$ was described by Savchuk, see \cite{Savchuk}. There
are two parts of the graph: the one that corresponds to the binary
tree and another is rays attached to every node of the tree. These
ray imitate positive part of the Cayley graph of $\mathbb{Z}$, and
we will call them hairs.\\

Let $K\subset Orb(1/2)$. Since $F$ is strongly transitive (see \cite{Brin}),
we can find an element $g$ that maps this set to the hair. We can
assume that this set is mapped deep enough into the hair. One the
hairs one of the generators, say $g_{0}$ (in the notations of Savchuk),
act as $\mathbb{Z}$. The set $g(K)$ might not be connected, but
it is clear that the uniform measure on $\{g_{0}^{k}:k\in[-n,n]\}$
will satisfy Lemma \ref{mainlemma} for sufficiently large $n$, therefore
the uniform measure on $\{g_{0}^{k}g:k\in[-n,n]\}$ is Liouville.
We can make it symmetric by taking $\{g^{-1}g_{0}^{k}g:k\in[-n,n]\}$
\end{proof}

\section{Liouville non-amenable Schreier graphs and lamplighters}

As another application of Lemma \ref{mainlemma} we show that certain
Schreier graphs of lamplighter group are non-amenable but $\mu$-Liouville
for some infinitely supported measure $\mu$. We note that examples
of non-amenable graphs with Liouville measures were previously discovered
in \cite{Ben-Koz}. There are many examples when a non-amenable group
$G$ admits an action on a set $X$ such that the induced action of
the lamplighter group $\bigoplus_{X}\mathbb{Z}/2\mathbb{Z}\rtimes G$
on $\bigoplus_{X}\mathbb{Z}/2$ is not amenable. In fact, if the action
of $G$ on $X$ is not amenable then the action of $\bigoplus_{X}\mathbb{Z}/2\mathbb{Z}\rtimes G$
on $\bigoplus_{X}\mathbb{Z}/2$, see for example \cite{JS}, \cite{JMMS}.
However, this action always admits Liouville measure.

\begin{lemma}

Consider the semi-direct product $G\ltimes A$ with $G$ discrete and
$A$ amenable. Then there exists a non-degenerate probability measure
$\mu$ on $G\ltimes A$ such that the action of $G\ltimes A$ on $A$
is $\mu$-Liouville.

\end{lemma}

\begin{proof}

Fix a sequence of finite subsets $\left(K_{n}\right)$ that exhaust
$A$. Since $A$ is amenable, we can find a sequence of measures $\left(\nu_{n}\right)$
on $A$ such that 
\[
\sup_{x,y\in K_{n}}\left\Vert \nu_{n}\cdot x-\nu_{n}\cdot y\right\Vert _{\ell^{1}(A)}\le\frac{1}{n}.
\]
Regard $\nu_{n}$ as a measure supported on $\{(e_{G},a):\ a\in A\}$,
then by Lemma \ref{mainlemma} we can find a non-degenerate measure
$\mu$ on $G\ltimes A$ such that $P_{\mu}$ is Liouville.

\end{proof}

Let $G$ be an amenable group. The every action of $G$ admits a Liouville
measure (possibly infinitely supported), moreover, the stabilizers
of the action are amenable. \\

While the following questions should have a negative answer, we currently
don't have any examples to support it.

\begin{question} Let $G$ act transitively on a set $X$ and assume
that this action is $\mu$-Liouville action of $G$ on $X$ for some
measure $\mu$ on $G$ such that $Stab_{G}(x)$ is abelian for some
(equivalently for all) $x$ in $X$. Is $G$ amenable? \end{question}

\begin{question} Let $G$ act transitively on a set $X$ and assume
that this action is $\mu$-Liouville action of $G$ on $X$ for some
measure $\mu$ on $G$. Denote by $X_{n}$ the set of all finite subsets
of size $n$. Then $G$ acts on $X_{n}$, however, this action may
not be transitive. Is it true that the action of $G$ on orbit of
$x\in X_{n}$ is Liouville for some measure? \end{question}


\begin{thebibliography}{10}
\bibitem{Brin}\textsc{Brin, M.}, \textit{The chameleon groups of
Richards J. Thompson: automorphisms and dynamics.} Publications Mathématiques
de l'IHES 84 (1996): 5-33.

\bibitem{Ben-Koz}\textsc{Benjamini, I., Kozma G}, \textit{Nonamenable
Liouville graphs.} arXiv preprint arXiv:1010.3365.

\bibitem{JS}\textsc{Juschenko, K., de la Salle, M.}, \textit{Invariant
means of the wobbling groups}.  arXiv preprint arXiv:1301.4736 (2013).

\bibitem{JMMS}\textsc{Juschenko, K., Matte Bon, N., Monod, N., de
la Salle, M.}, \textit{Extensive amenability and an application to
interval exchanges.} arXiv preprint arXiv:1503.04977.

\bibitem{Kaimanovich} \textsc{Kaimanovich, V.}, \textit{Thompson's
group $F$ is not Liouville}, Preprint, arXiv:1602.02971

\bibitem{KaimVersh} \textsc{Kaimanovich, V. A., Vershik, A. M.},
\textit{Random walks on discrete groups: boundary and entropy}, The
Annals of Probability, (1983) 457-490.

\bibitem{LP:book} \textsc{Lyons R., Peres Y.}, \textit{{Probability
on Trees and Networks}, Cambridge University Press, 2016.}

\bibitem{Mishchenko}\textsc{Mishchenko, P.}\textit{, Boundary
of the action of Thompson group F on dyadic numbers, Preprint, arXiv:1512.03083}

\bibitem{Savchuk}\textsc{Savchuk, D.}\textit{, Schreier
graphs of actions of Thompson's group F on the unit interval and on
the Cantor set, Geometriae Dedicata, V.175 (2015), 355--372}\end{thebibliography}
\end{document}